\documentclass[reqno,a4paper,twosided]{amsart}
 
\usepackage[utf8]{inputenc}
\usepackage{amsmath}
\usepackage[usenames,dvipsnames]{xcolor}
\usepackage[colorlinks]{hyperref}
\usepackage{amssymb}
\usepackage{amsthm}
\usepackage{latexsym}
\usepackage{enumerate}
\usepackage{enumitem}
\usepackage{mathtools}
\usepackage{mathrsfs}
\usepackage{microtype}
\usepackage[font={footnotesize}]{caption}
\usepackage{subcaption}
\usepackage[all]{xy}
\usepackage{etoolbox}
\usepackage{cleveref}
\usepackage[foot]{amsaddr}
\usepackage{graphicx}
\usepackage{dsfont}
\usepackage{fullpage}
\usepackage{scalerel}

\newtheorem{theorem}{Theorem}
\newtheorem{lemma}[theorem]{Lemma}
\newtheorem{cor}[theorem]{Corollary}
\newtheorem{prop}[theorem]{Proposition}

\newtheorem*{questionA}{Question A}
\newtheorem*{questionB}{Question B}

\theoremstyle{definition}

\newtheorem{remark}[theorem]{Remark}
\newtheorem{exm}[theorem]{Example}

\Crefname{theorem}{Theorem}{Theorems}
\Crefname{lemma}{Lemma}{Lemmata}
\Crefname{cor}{Corollary}{Corollaries}
\Crefname{prop}{Proposition}{Propositions}
\Crefname{defn}{Defintion}{Definitions}
\Crefname{remark}{Remark}{Remarks}
\Crefname{exm}{Example}{Examples}

\numberwithin{equation}{section} 
\numberwithin{theorem}{section}
\numberwithin{figure}{section}

\hypersetup{
	allcolors = Red,
	allbordercolors = Red,
	colorlinks = true}

\newcommand{\bN}{\mathbb{N}} 
\newcommand{\bQ}{\mathbb{Q}} 
\newcommand{\bZ}{\mathbb{Z}} 
\newcommand{\bC}{\mathbb{C}} 
\newcommand{\bR}{\mathbb{R}} 
\newcommand{\Per}{\operatorname{Preper}}
\renewcommand{\P}{\operatorname{Per}}

\makeatletter
\@ifdefinable\@latex@chi{\let\@latex@chi\chi}
\renewcommand*\chi{{\@latex@chi\smash[t]{\mathstrut}}} 
\makeatletter

\usepackage[colorinlistoftodos,textsize=footnotesize,backgroundcolor=Green!50,bordercolor=Green,linecolor=Green]{todonotes}
\tikzset{/tikz/notestyleraw/.append style={text=black}}
\begin{document}

\title{Periodic intermediate $\beta$-expansions of Pisot numbers}

\author[B.\,Quackenbush]{Blaine Quackenbush$^{(1)}$}
\author[T.\,Samuel]{Tony Samuel$^{(2)}$}
\author[M.\,West]{Matt West$^{(3)}$}

\address{$^{(1)}$ Mathematics Department, California Polytechnic State University, CA, USA}
\address{$^{(2)}$ School of Mathematics, University of Birmingham, UK}
\address{$^{(3)}$ Department of Mathematics, University of California, Irvine, CA, USA\vspace{0.25em}}

\address{\parbox[t]{0.97\textwidth}{\textit{E-mail addresses}. \texttt{bquacken@calpoly.edu}, \texttt{t.samuel@bham.ac.uk}, \texttt{mawest@uci.edu}}\vspace{0.25em}}

\keywords{$\beta$-expansions/transformations, shifts of finite type, periodic points, iterated function systems.\vspace{0.5em}}

\subjclass[2010]{Primary: 37E05, 37B10; Secondary: 11A67, 11R06\vspace{0.25em}}

\thanks{\parbox[b]{0.96\textwidth}{\textit{Acknowledgements}. Part of the work presented here was carried out as an undergraduate research project with some of our students, notably the first and third authors.  We are extremely grateful to the Bill and Linda Frost Fund which partially supported this work.  The second author would also like to thank the Hausdorff Research Institute for Mathematics and Institut Mittag Leffler for their kind hospitality during the final writing stages of this article.}\hspace{-6.25em}}

\begin{abstract}
The subshift of finite type property (also known as the Markov property) is ubiquitous in dynamical systems {and the} simplest and most widely studied class of {dynamical systems are $\beta$-shifts, namely transformations of the form $T_{\beta, \alpha}  \colon x \mapsto \beta x + \alpha \bmod{1}$ acting on $[-\alpha/(\beta - 1), (1-\alpha)/(\beta - 1)]$}, where $(\beta, \alpha) \in \Delta$ is fixed and where  $\Delta \coloneqq \{ (\beta, \alpha) \in  \bR^{2} \colon \beta \in (1,2) \; \text{and} \; 0 \leq \alpha \leq 2-\beta \}$.  Recently, it was shown, by Li \textsl{et al.} (\textsl{Proc.\ Amer.\ Math.\ Soc.} 147(5):\ 2045--2055, 2019), that the set {of} $(\beta, \alpha)$ such that $T_{\beta, \alpha}$ has the subshift of finite type property is dense in the parameter space $\Delta$.  Here, they proposed the following question.  Given a fixed $\beta \in (1, 2)$ which is the $n$-th root of  a Perron number, does there exists a dense set of $\alpha$ {in the fiber $\{\beta\} \times (0, 2- \beta)$,} so that $T_{\beta, \alpha}$ has the subshift of finite type property?\\[0.5em]
\noindent We answer this question in the positive for a class of Pisot numbers.  Further, we investigate if this question holds true when replacing the subshift of finite type property by the property of beginning sofic (that is a factor of a subshift of finite).  In doing so we generalise, a classical result of Schmidt (\textsl{Bull.\ London Math.\ Soc.}, 12(4):\ 269--278, 1980) from the case when $\alpha = 0$ to the case when $\alpha \in (0, 2- \beta)$.  That is, we examine the structure of the set of eventually periodic points of $T_{\beta, \alpha}$ when $\beta$ is a Pisot number and when $\beta$ is the $n$-th root of a Pisot number.
\end{abstract}

\maketitle

\section{Introduction and Statement of Main Results}

\subsection{Introduction}

Since the pioneering work of R\'enyi \cite{R:1957} and Parry \cite{MR0142719}, {$\beta$-shifts} and expansions have been extensively studied and have provided practical solutions to various problems.  For instance, {they} arise as Poincar\'e maps of the geometric model of Lorenz differential equations \cite{MR681294}, and Daubechies \textsl{et al.} \cite{1011470} proposed a new approach to analog-to-digital conversion using $\beta$-expansion.  A summary of some further applications can be found in \cite{LSSW}.  Through their study, many new phenomena have appeared, revealing a rich combinatorial and topological structure, and unexpected connections to probability theory, ergodic theory, number theory and aperiodic order \cite{Komornik:2011,ArneThesis,MR2052279}.   Additionally, through understanding {$\beta$-shifts} and expansions, advances have been made in the theory of Bernoulli convolutions \cite{AFPK,MR3084706}.

For $\beta > 1$ and $x \in [0,1/(\beta-1)]$, a word $(\omega_{n})_{n \in \bN}$ in the alphabet $\{ 0, 1\}$ is called a $\beta$-\textsl{expansion} of $x$ if
\begin{align*}
x = \sum_{k= 0}^{\infty} \omega_{k+1} \ \beta^{-k}.
\end{align*}
When $\beta$ is a natural number, all but a countable set of real numbers have a unique \mbox{$\beta$-expansion}.  On the other hand, in \cite{MR1078082}, it was shown that, if $\beta$ is less than the golden mean, then for all $x \in (0, 1/(\beta-1))$, the cardinality of the set of $\beta$-expansions of $x$ is equal to the cardinality of the continuum.  Siderov \cite{Sidorov:2003} extended this result and showed that if $\beta$ is {strictly} less than two, then for Lebesgue almost all $x \in [0, 1/(\beta-1)]$, the cardinality of the set of $\beta$-expansions of $x$ equals the cardinality of the continuum.

Through iterating the maps $G_{\beta} \colon [0,1/(\beta-1)] \circlearrowleft$ and $L_{\beta} \colon [(\beta-2)/(\beta - 1), 1] \circlearrowleft$ defined by
	\begin{align*}
	G_{\beta}(x) \coloneqq 
	\begin{cases}
	\beta x  & \text{if} \; x < 1/\beta,\\
	\beta x -1 & \text{otherwise,}
	\end{cases}
%
\quad \text{and} \quad
%
	L_{\beta}(x)  \coloneqq 
	\begin{cases}
	\beta x + 2 - \beta  & \text{if} \; {x \leq 1 - 1/\beta,}\\
	\beta x + 1 - \beta & \text{otherwise.}
	\end{cases}
	\end{align*}
one obtains subsets of $\{0, 1\}^{\bN}$ known as the greedy and lazy $\beta$-shifts, respectively, where each point $\omega^{+}$ of the greedy $\beta$-shift and each point $\omega^{-}$ of the lazy $\beta$-shift corresponds to a $\beta$-expansion of a unique point in the interval $[0,1/(\beta-1)]$.  Note, if $\omega^{+}$ and $\omega^{-}$ are $\beta$-expansions of the same point, then $\omega^{+}$ and $\omega^{-}$ are not necessarily equal, see \Cref{ex:inbetween} and \cite{MR2299792,Komornik:2011,Komornik:1998}.

There are many ways, other than using the greedy and lazy $\beta$-shift, to generate a \mbox{$\beta$-expansion} of a real number. For instance, from intermediate $\beta$-shifts $\Omega_{\beta, \alpha}^{\pm}$, which arise from \textsl{intermediate $\beta$-transformations} $T^{\pm}_{\beta,\alpha} \colon [ -\alpha/(\beta - 1), (1-\alpha)/(\beta - 1) ] \circlearrowleft$, where $(\beta,\alpha) \in \Delta \coloneqq \{ (b, a) \in \mathbb{R}^{2} \colon b \in (1, 2) \; \text{and} \; a \in [0, 2 - \beta] \}$ and where $T_{\beta, \alpha}^{\pm}$ are defined as follows.  Letting $p = p_{\beta, \alpha} \coloneqq (1-\alpha)/\beta$ we set
\begin{align*}
T^{+}_{\beta, \alpha}(x) \coloneqq 
\begin{cases}
\beta x + \alpha & \text{if} \; x < p,\\
\beta x + \alpha - 1 & \text{otherwise},
\end{cases}
\quad \text{and} \quad
T^{-}_{\beta, \alpha}(x) \coloneqq 
\begin{cases}
\beta x + \alpha & \text{if} \; x \leq p,\\
\beta x + \alpha -  1 & \text{otherwise}.
\end{cases}
\end{align*}
The maps $T_{\beta, \alpha}^{\pm}$ are equal everywhere except at $p$ and $T^{-}_{\beta, \alpha} (x) = 1 - T_{\beta, 2 - \beta - \alpha}^+(1 - x)$. Notice, when $\alpha = 0$, the maps $G_{\beta}$ and $T^{+}_{\beta, \alpha}$ coincide, and when $\alpha = 2-\beta$, the maps $L_{\beta}$ and $T^{-}_{\beta, \alpha}$ coincide.  Further, observe that $-\alpha/(\beta - 1)$ and $(1-\alpha)/(\beta - 1)$ are fixed points for $T_{\beta, \alpha}^{\pm}$ and that {the unit interval} $[0, 1]$ is a trapping {region} for $T_{\beta, \alpha}^{\pm}$, meaning that if $x \in [0, 1]$, then ${(T_{\beta,\alpha}^{\pm})^{n}(x)}  \in [0, 1]$, for all $n  \in \bN$; and if $x \in ( -\alpha/(\beta - 1), 0) \cup (1, (1-\alpha)/(\beta - 1) )$, then there exists an $m \in \bN$ such that  {$(T_{\beta,\alpha}^{\pm})^{m}(x)  \in [0, 1]$}.

Each point in $\Omega_{\beta, \alpha}^{\pm}$ is a $\beta$-expansion of a unique point in $[0, 1/(\beta - 1)]$, see \eqref{diag:commutative2}, and $\Omega_{\beta, \alpha} \coloneqq \Omega_{\beta, \alpha}^{+} \cup \Omega_{\beta, \alpha}^{-}$ is a subshift, meaning that it is invariant under the (left) shift map $\sigma$ and closed in $\{ 0, 1 \}^{\bN}$, where we equip $\{ 0, 1 \}$ with the discrete topology and $\{ 0, 1 \}^{\bN}$ with the product topology.  The dynamical systems $(\Omega^{\pm}_{\beta, \alpha}, \sigma)$ and $([0, 1], T^{\pm}_{\beta, \alpha})$ are topologically conjugate,  that is they have `the same' dynamical properties.

Subshifts which can be completely described by a finite set of forbidden words are called \textsl{subshifts of finite type} (see \Cref{sec:Prelim-Subshift}) and play an essential {r\^ole} in the study of dynamical systems.  A reason why subshifts of finite type are so useful is that they have a simple representation as a finite directed graph.  Thus, {dynamical and combinatorical} questions about the subshift can be phrased in terms of an adjacency matrix making them much more tractable.  Hence, it is of interest to classify the set of $(\beta, \alpha) \in \Delta$ for which $\Omega_{\beta, \alpha}$ a subshift of finite type.  One of our aims is to give new insights towards such a classification.

Given $(\beta, \alpha) \in \Delta$, the unique points in $\Omega_{\beta, \alpha}^{+}$ and $\Omega_{\beta, \alpha}^{-}$ corresponding to $p$ are called the \textsl{kneading invariants} of $\Omega_{\beta, \alpha}$.  It is known that the kneading invariants completely determine $\Omega_{\beta, \alpha}$, see \Cref{omegastructure} due to \cite{BHV,H,HS}, and the {$\beta$-shift} $\Omega_{\beta, \alpha}$ is a subshift of finite type if and only if the left shift of the kneading invariants are periodic, see \Cref{Parry:GL} due to Ito and Takahashi \cite{MR0346134}, and Parry \cite{Parry:1979}, for the case $\alpha \in \{0, 2- \beta\}$, and Li \textsl{et al.} \cite{LSS}, for the case that $\alpha \in (0, 2- \beta)$.  These results immediately give us that the set of parameters in $\Delta$ which give rise to $\beta$-shifts of finite type is countable.  In a second article \cite{LSSW} by Li \textsl{et al.},  it was shown that this set of parameters is in fact dense in $\Delta$.  In contrast, if one considers the dynamical property of topologically transitivity, then the structure of the set of $(\beta, \alpha)$ in $\Delta$ such that $\Omega_{\beta, \alpha}$ is topologically transitive, with respect to the left shift map, is very different to the set of $(\beta, \alpha)$ belonging to $\Delta$ for which $\Omega_{\beta, \alpha}$ is a subshift of finite type.  It is worth noting that the former of these two sets has positive Lebesgue measure and is far from being dense in $\Delta$, see \Cref{thm:Palmer} due to Palmer \cite{Par:1979} and Glendinning \cite{G:1990}.

The results of \cite{MR1399483} and \cite{Lind:84} in tandem with those discussed above, yield the following.
	\begin{enumerate}[label={(\roman{enumi})}]
	\item If the $\beta$-shift $\Omega_{\beta, \alpha}$ is a subshift of finite type, then $\alpha \in \bQ(\beta)$. 
	\item If $\beta$ is not the positive $n^{\textup{th}}$-root of a Perron number, for some $n \in \bN$, then the set of $\alpha$ for which the $\beta$-shift $\Omega_{\beta, \alpha}$ is a subshift of finite type is empty.
	\end{enumerate}
Indeed, for $\Omega_{\beta, \alpha}$ to be a subshift of finite type, we require $\beta \in (1, 2)$ to be a maximal root of a polynomial with coefficients in $\{ -1, 0, 1 \}$ {and $\alpha \in \bQ(\beta)$}.  This leads to the following natural question, to which we give a partial answer to in \Cref{SFTdense}.  

\begin{questionA}\label{questionA}
If $\beta \in (1, 2)$ is a positive $n^{\textup{th}}$-root of a Perron number, for some $n \in \bN$, is the set of $\alpha$ for which $\Omega_{\beta, \alpha}$ is a subshift of finite type dense in $(0, 2 - \beta)$?
\end{questionA}

Another class of subshifts which is of interest here are those which are factors of a subshift of finite type.  Such subshifts are called \textsl{sofic}; indeed, every subshift of finite type is sofic, but not vise versa.  Kalle and Steiner \cite{KalleSteiner} proved that a $\beta$-shift $\Omega_{\beta, \alpha}$ is sofic if and only if its kneading invariants are eventually periodic.  Combining this result with those of Li \textsl{et al.} \cite{LSSW}, one obtains that the set of $(\beta, \alpha) \in \Delta$ for which $\Omega_{\beta, \alpha}$ is sofic is dense in $\Delta$.  This naturally leads to the study of (eventually) periodic points.

Bertrand \cite{MR0447134} and Schmidt  \cite{KS}, and subsequently Boyd \cite{MR1024551,MR1333306,MR1483916} and  Maia \cite{Maia_18}, addressed the following question.  For a fixed $\beta$, what are the values of $x \in [0, 1]$ which are eventually periodic under $G_{\beta}$? Recall, a point $x$ is eventually periodic under $G_{\beta}$ if and only if the cardinality of the set $\{ G_{\beta}^{n}(x) \colon n \in \bN \}$ is finite. Letting $\Per(\beta)$ denote the set of $x$ which are eventually periodic under $G_{\beta}$, Schmidt points out that if $x, y \in  \Per(\beta) \cap [0, 1]$, then there is no obvious reason why $x + y \bmod{1}$ should also be an element of  $\Per(\beta) \cap [0, 1]$. In view of this it is surprising that certain $\beta > 1$ behave exactly like integers, in the sense that if $\beta$ is a Pisot number, then $\Per(\beta) \cap [0, 1] = \bQ(\beta) \cap [0, 1]$.  A natural question  to ask here is:

\begin{questionB}\label{questionB}
What is the structure of the set $\Per^{\pm}(\beta, \alpha)$ of eventually periodic points under $T_{\beta, \alpha}^{\pm}$?
\end{questionB}

In \Cref{thm:thmB}, we show, if $\beta$ is a Pisot number and $\alpha \in \bQ(\beta) \cap (0, 2 - \beta)$, then $\Per^{\pm}(\beta,\alpha) = \bQ(\beta) \cap J_{\beta, \alpha}$, where $J_{\beta, \alpha}$ denotes the domain of $T_{\beta, \alpha}^{\pm}$.  We also obtain a partial converse and in \Cref{cor:sofic}, we relate these results back to Question A.

\subsection{Statement of main results}

Our main contributions in this article and to the story of periodic $\beta$-expansions, is to show the following results, namely \Cref{SFTdense,thm:thmB}, and \Cref{cor:sofic,thm:thmC,thm:thmD}.

For $\beta \in (1, 2)$ set $\Delta(\beta) \coloneqq \{ (\beta, \alpha) \in \mathbb{R}^{2} \colon 0 \leq \alpha \leq 2 - \beta \}$ and recall that the \textsl{multinacci number} $\beta_{m}$ \textsl{of order} $m \geq 2$ is the unique real solution to the equation $x^{m} = x^{m-1} + \dots + x + 1$ in the interval $(1,2)$. Note, the sequence $(\beta_m)_{m = 2}^{\infty}$ is strictly increasing and converges to $2$, and that $\beta_2$ is the golden mean.

\begin{theorem} \label{SFTdense}
Fix $m \geq 2$ an integer. The set of $(\beta_{m}, \alpha)$ in $\Delta(\beta_{m})$ with $\Omega_{\beta_{m}, \alpha}$ a subshift of finite type is dense in $\Delta(\beta_m)$.
\end{theorem}

The main difficulty in proving \Cref{SFTdense} was in finding a way to compare the space $\Omega_{\beta_{m}, \alpha}$ and $\Omega_{\beta_{m}, \alpha'}$, for a fixed $m$ and $\alpha \neq \alpha'$.  We achieved this by embedding all $\beta_{m}$-transformations into a single (multi-valued) dynamical system and carrying out our analysis in this larger system.

This result answers Question A for the class of multinacci number which belong to the wider class of algebraic numbers known as Pis\^{o}t numbers.  Although many parts of our proof generalise from the class of multinacci numbers to the class of Pisot numbers, a central result (\Cref{+iff-}) which state that the upper kneading invariant is periodic if any only if the lower kneading invariant is periodic does not easily generalise, see \Cref{ex:+periodic-non-periodic} for an example of a point $(\beta, \alpha) \in \Delta$ where this is not the case.  Here we would like to mention that \Cref{+iff-} is closely related to the property known as matching, which has has been extensively studied \cite{KalleBruinCarminati,0951-7715-32-1-172}.

In the hope of circumventing this we turn our attention to Question B and examined the set of eventually periodic points under $T_{\beta, \alpha}^{\pm}$.

\begin{theorem}\label{thm:thmB}
Let $\beta \in (1,2)$ and $\alpha \in \bQ(\beta) \cap (0, 2-\beta)$ be fixed, and let $J_{\beta, \alpha}$ denote the domain of $T_{\beta, \alpha}^{\pm}$.
\begin{enumerate}[label={(\roman{enumi})}]
\item\label{thm:thmB1} If $\bQ \cap J_{\beta, \alpha} \subseteq \Per^{\pm}(\beta,\alpha)$, then $\beta$ is either a Pis{o}t or a Salem number.
\item\label{thm:thmB2} If $\beta$ is a Pisot number, then $\Per^{\pm}(\beta,\alpha) = \bQ(\beta) \cap J_{\beta, \alpha}$.
\end{enumerate}
\end{theorem} 
 
\Cref{thm:thmB}\,\ref{thm:thmB1} also hold when $J_{\beta, \alpha}$ is replaced by $[0, 1]$, since $[0, 1]$ is a trapping region for $T_{\beta, \alpha}$.

As indicated above,  \Cref{thm:thmB} generalises the  results of Schmidt \cite{KS}.  Indeed, our proof is motivated by that of \cite{KS}, with the following crucial difference.  In the setting of \cite{KS}, namely when $\alpha = 0$, a key fact that is used is to any point $x$ there exists a point $y$ arbitrarily close to $x$ and integers $m$  and $n$, such that $G^{n + k}_{\beta}(y)$ is arbitrarily close to zero for all $k \in \{ 0, 1, \dots, m \}$.  However, this is not the case, when $\alpha > 0$.  To circumvent this, we appeal to the kneading theory of Milnor and Thurston discussed in \Cref{sec:betashifts}.  We also remark that a similar question to Question B was consider by Baker  \cite{MR3683942}; via different methods to ours, and also Schmidt's, \Cref{thm:thmB}\,\ref{thm:thmB1} maybe concluded from the work of Baker and \Cref{thm:thmB}\,\ref{thm:thmB2} can be seen as a strengthening of Baker's results.

Further, as a consequence of \Cref{thm:thmB}\,\ref{thm:thmB2} and a result of \cite{KalleSteiner}, see \Cref{thm:sofic}, we obtain the following partial solution to Question A.

\begin{cor} \label{cor:sofic}
Let $\beta \in (1,2)$ be a Pisot number. The set of $(\beta, \alpha)$ in $\Delta(\beta)$ for which $\Omega_{\beta,\alpha}$ is sofic is dense in $\Delta(\beta)$.
\end{cor}

In addition to this, combining the results of Palmer \cite{Par:1979} and Glendinning \cite{G:1990} as well as Parry \cite{Par:1979,MR0166332} with \Cref{SFTdense,thm:thmB}, we may 
\begin{enumerate}[label={(\roman{enumi})}]
\item determine a set of $\alpha$ which lie dense in a subset of positive Lebesgue measure of the fibre $\Delta(\beta_{m}^{1/n})$, for all integers $m$ and $n \geq  2$, and 
\item classify the set $\Per(\beta, \alpha)$, in the case that $\beta$ is the $n$-th root of a Pisot number and $T_{\beta, \alpha}$ is non-transitive.
\end{enumerate}
In  order to state these results we require a few preliminaries.

Let $n$ and $k \in \bN$ with $k < n$ and $\operatorname{gcd}(n, k) = 1$ be given, and let $s \in \{ 0, 1, \dots, k-1\}$ be such that $n = s \bmod{k}$.  For $j \in  \{ 1, 2, \dots, s \}$, define $V_{j}$ and $r_{j}$ by $jk = V_{j} s + r_{j}$, where $r_{j} \in \{ 0, 1, \dots, s - 1 \}$, and $h_{j}$ by $V_{j} = h_{1} + h_{2} + \dots + h_{j}$.  For $\beta \in (1, 2^{1/n}]$ set
	\begin{align}\label{eq:Ink}
	\begin{aligned}
	I_{n, 1}(\beta) &\coloneqq \left[ \frac{1}{\beta (\beta^{n-1} + \dots + 1)}, \frac{- \beta^{n+1} + \beta^{n} + 2\beta - 1}{\beta(\beta^{n-1} + \dots + 1)} \right],\; \text{and}\\[0.5em]
	I_{n, k}(\beta) &\coloneqq \left[ \frac{1 + \beta (\sum_{j=1}^{s} W_{j} - 1) }{\beta (\beta^{n-1} + \dots + 1)}, \frac{\beta(\sum_{j=1}^{s} W_{j}) - \beta^{n+1} + \beta^{n} + \beta - 1}{\beta(\beta^{n-1} + \dots + 1)} \right],
	\end{aligned}
	\end{align}
where, for $2 \leq j \leq s$,
\begin{align}\label{eq:wj}
W_{j} \coloneqq \sum_{i = 1}^{h_{j}} \beta^{(V_{s} - V_{j-1} - i)m + s - j}
\quad \text{and} \quad
W_{1} \coloneqq \sum_{i = 1}^{V_{1}} \beta^{(V_{s} - i)m + s - 1}
\end{align}
see \Cref{Fig:ParPlot} for a sketch of the intervals $I_{n, k}(\beta)$.  If $\beta = 2^{1/n}$, then $I_{n, k}(\beta)$ is a single point and, if $\beta \in (0, 2^{1/n})$, then $I_{n, k}(\beta)$ is an interval of positive Lebesgue measure.  Further, for a fixed $\beta \in (1, 2)$, in \cite{G:1990}, it was shown that the Lebesgue measure of
\begin{align*}
\left\{ \alpha \in (0, 2- \sqrt[l]{\beta}) \colon k \in \{ 1, \dots, l\} \; \text{with} \; \operatorname{gcd}(l, k) = 1 \; \text{and} \; \alpha \in I_{l, k}(\sqrt[l]{\beta}) \right\}
\end{align*}
remains bounded away from zero as $l \in \mathbb{N}$ tends to infinity.

\begin{cor}\label{thm:thmC}
Let $m$ and $n  \geq 2$ denote two natural numbers, and let $k \in \bN$ be such that $k < n$ and $\operatorname{gcd}(n, k) = 1$.  There exists a dense set of $\alpha$ in $I_{n, k}(\sqrt[n]{\beta_{m}})$ with $\Omega_{\sqrt[n]{\beta_{m}}, \alpha}$ a subshift of finite type.  Moreover, if $\beta$ is a Pisot number, then there exists a dense set of $\alpha$ in $I_{n, k}(\sqrt[n]{\beta})$ with $\Omega_{\sqrt[n]{\beta}, \alpha}$ sofic.
\end{cor}

Before stating our final corollary we require one last preliminary.  For $(\beta, \alpha) \in \Delta$, Parry \cite{MR0166332} constructed an absolutely continuous $T_{\beta, \alpha}$-invariant probability measure, which we denote by $\nu_{\beta, \alpha}$, and in \cite{MR0386019}, it was verified that the density $h_{\beta, \alpha}$ is always non-negative.  Hofbauer \cite{H,MR570882,MR599481} showed that this measure is ergodic and maximal, and a direct consequence of \cite{Par:1979} and \cite{G:1990} is that $\nu_{\beta, \alpha}$ has support equal to $[0, 1]$ if and only if $T_{\beta, \alpha}$ is topologically transitive.

\begin{cor}\label{thm:thmD}
Let $n$ and $k \in \bN$, let $\beta$ denote a Pisot number and let $\alpha \in \bQ(\beta) \cap [0, 2 - \beta]$. Defining $\alpha_{n, k} = \alpha_{n, k}(\beta, \alpha) \in I_{n, k}(\sqrt[n]{\beta})$ by
	\begin{align*}
	\alpha_{n, 1} \coloneqq \frac{((1 - \alpha)(1 - \sqrt[n]{\beta^{-1}}) - 1)(1 - \sqrt[n]{\beta})}{\beta - 1}
	\quad \text{and} \quad
	\alpha_{n, k} \coloneqq \frac{((1 - \alpha)(1 - \sqrt[n]{\beta^{-1}}) - \sum_{j  = 1}^{s} W_{j})(1 - \sqrt[n]{\beta})}{\beta - 1},
	\end{align*}
where $s \in \{ 0, 1, \dots, k-1\}$ satisfies $n = s \bmod{k}$ and $W_{j}$ is as in \eqref{eq:wj}, and setting $\Phi(x) \coloneqq (\sqrt[n]{\beta} - 1)x + \alpha_{n, k}$, we have
	\begin{align*}
	\Per^{\pm}(\sqrt[n]{\beta}, \alpha_{n, k}) \cap \operatorname{supp}(\nu_{\sqrt[n]{\beta}, \alpha_{n, k}}) = \bigcup_{i = 0}^{n-1} (T_{\sqrt[n]{\beta}, \alpha_{n, k}}^{\pm})^{i} (\Phi( \bQ(\beta) \cap [0, 1]) ).
	\end{align*}
\end{cor}

\subsection{Outline} 

In \Cref{sec:Prelim} we give necessary definitions and results we require in our proofs of \Cref{SFTdense,thm:thmB}.  \Cref{sec:Proof_Thm_1,sec:Proof_Thm_2_&_3} are dedicated to proving \Cref{SFTdense,thm:thmB}, respectively.  We conclude with \Cref{sec:transitivity}.  The aim of this final section is to provide an overview of the results of \cite{G:1990,Par:1979,MR0166332,Parry:1979} which in combination with our results (\Cref{SFTdense,thm:thmB}) yields \Cref{thm:thmC,thm:thmD}.

\section{Preliminaries}\label{sec:Prelim}

We divide this section into three parts: \Cref{sec:Prelim-Subshift,sec:betashifts} in which we discuss aspects of symbolic dynamics and $\beta$-shifts; and \Cref{sec_uniform} where we review results concerning a related class of interval maps, namely uniform Lorenz maps, which are in essence scaled versions of $\beta$-transformations.

\subsection{Subshifts}\label{sec:Prelim-Subshift}

We equip the set $\{0,1\}^\bN$ of infinite words with the topology induced by the ultra metric $\mathscr{D} \colon \{0,1\}^\bN \times \{0,1\}^\bN \to \mathbb{R}$ defined by
\begin{align*}
\mathscr{D}(\omega, \nu) = 
\begin{cases}
0 & \text{if} \; \omega = \nu,\\
2^{- \lvert  \omega \wedge \nu \rvert + 1} & \text{if} \; \omega \neq \nu,
\end{cases}
\end{align*}
where $\lvert \omega \wedge \nu \rvert \coloneqq \min \{ i \in \bN \colon \omega_i \neq \nu_i \}$, for $\omega = (\omega_1, \omega_2, \dots)$ and $\nu = (\nu_1, \nu_2, \dots)$. This topology coincides with the product topology on $\{0,1\}^\bN$, where $\{0,1\}$ is endowed with the discrete topology. For $n \in \bN$ and $\omega \in \{0,1\}^\bN$, we set $\omega\vert_n = (\omega_1, \dots, \omega_n)$ and call $n$ the \textsl{length} of $\omega\vert_n$ denoted by $\lvert \omega\vert_n \rvert$.  We define the (\textsl{left}) \textsl{shift} $\sigma$ on $\{0,1\}^\bN$ by $\sigma(\omega_1, \omega_2, \dots) \coloneqq (\omega_2, \omega_3, \dots)$. A closed subspace $\Omega$ of $\{0,1\}^\bN$ is a \textsl{subshift} if $\Omega$ is invariant under $\sigma$, namely $\sigma( \Omega) \subseteq \Omega$. Given a subshift $\Omega$ and a natural number $n$, we set
\begin{align*}
\Omega\vert_n \coloneqq \{(\xi_1, \xi_2, \dots, \xi_n) \in \{0,1\}^n \colon \text{there exists} \; \omega \in \Omega \; \text{with} \; \omega\vert_{n} = (\xi_1, \xi_2, \dots, \xi_n)\}
\end{align*}
and denote by $\Omega^* \coloneqq \bigcup_{n \in \bN} \Omega\vert_n$ the collection of all finite words.  A subshift $\Omega$ is said to be \textsl{of finite type} if there exists a finite set $F$ of finite words such that 
	\begin{enumerate}[label={(\roman{enumi})}]
	\item $\nu\vert_{n} \not\in  F$ for all $\nu \in \Omega$ and $n \in \bN$;
	\item if $\nu \in \{ 0, 1\}^{\bN} \setminus \Omega$, then there exist integers $n > 0$ and $m \geq  0$ such that $\sigma^{m}(\nu)\vert_{n} \in F$.
	\end{enumerate}
The set $F$ is often referred to as the \textsl{set of forbidden words} of $\Omega$.  If $\Omega \subseteq \{ 0, 1\}^{\bN}$ is a factor of a subshift of finite type, then it is called \textsl{sofic}.

A word $\omega$ is  \textsl{periodic} with period $n \in \bN$, if $(\omega_1, \dots, \omega_n) = (\omega_{(m-1)n+1}, \dots, \omega_{mn})$ for all $m \in \bN$; in which case we write $\omega = (\overline{\omega_1, \dots, \omega_n})$. The smallest such $n$ is called the period of $\omega$.   Similarly, a word $\omega$ is called \textsl{pre-periodic} with period $n \in \bN$, if there exists $k \in \bN$ with $(\omega_{k + 1}, \dots, \omega_{k + n}) = (\omega_{k + (m-1)n+1}, \dots, \omega_{k + mn})$ for all $m \in \bN$; in which case we write $\omega= (\omega_1, \dots, \omega_k, \overline{\omega_{k+1}, \dots, \omega_{k+n}})$.  As with periodic words, the smallest such $n$ is called the period of $\omega$.

\subsection{Intermediate $\beta$-shifts and expansions}\label{sec:betashifts}

Let $(\beta, \alpha) \in \Delta$ be fixed and set $p = p_{\beta,\alpha} \coloneqq (1 - \alpha)/\beta$. Let $\tau_{\beta, \alpha}^{\pm} \colon J_{\beta, \alpha} \to \{ 0, 1 \}^{\bN}$ be defined by 
\begin{align*}
\tau_{\beta,\alpha}^\pm(x) \coloneqq (\omega_1^\pm(x), \omega_2^\pm(x), \dots),
\end{align*}
where, for $n \in \bN$,
\begin{align*}
\omega_n^+(x) \coloneqq 
\begin{cases}
0 & \text{if} \; (T_{\beta,\alpha}^+)^{n-1}(x) < p,\\
1 & \text{otherwise,} 
\end{cases}
\quad \text{and} \quad
\omega_n^-(x) = 
\begin{cases}
0 & \text{if} \; (T_{\beta,\alpha}^-)^{n-1}(x) \leq p, \\
1 & \text{otherwise.} 
\end{cases}
\end{align*}
We refer to $\tau_{\beta, \alpha}^{\pm}$ as \textsl{expansion maps}.  The image of $J_{\beta, \alpha}$ under $\tau_{\beta,\alpha}^\pm$ is denoted by $\Omega_{\beta,\alpha}^\pm$, and we set $\Omega_{\beta, \alpha} \coloneqq \Omega_{\beta,\alpha}^+ \bigcup \Omega_{\beta,\alpha}^-$.   We call $\tau_{\beta,\alpha}^+(p)$ the \textsl{upper} and $\tau_{\beta,\alpha}^-(p)$ the \textsl{lower kneading invariant} of $\Omega_{\beta,\alpha}$.

\begin{remark}\label{criticalitinerarybegin}
Let $\omega = (\omega_1, \omega_2, \dots)$ and $\nu=(\nu_1, \nu_2, \dots)$ respectively denote the upper and the lower kneading invariant of $\Omega_{\beta, \alpha}$. By definition, $\omega_1 = \nu_2 = 0$ and $\omega_2 = \nu_1 = 1$. It can also be shown, for $k \geq 2$ an integer, that $(\omega_k, \omega_{k+1}, \dots) = (\overline{1})$ if and only if $\alpha = 2-\beta$, and that $(\nu_k, \nu_{k+1}, \dots) = (\overline{0})$ if and only if $\alpha = 0$.
\end{remark}

The inverse map $\pi_{\beta, \alpha}\colon \{0,1\}^\bN \to J_{\beta, \alpha}$ of $\tau_{\beta, \alpha}^{\pm}$ is called the \textsl{projection map} and defined by 
\begin{align*}
\pi_{\beta, \alpha}(\omega_{1}, \omega_{2}, \dots) \coloneqq \alpha(1-\beta)^{-1} + \sum_{i=1}^{\infty} \omega_i\beta^{-i}.
\end{align*}
An important property of $\tau_{\beta, \alpha}^{\pm}$ and $\pi_{\beta, \alpha}$ is that the following diagram commutes.
	\begin{align}\label{diag:commutative2}
	\begin{aligned}
	\xymatrix@C+2pc{
	\Omega^{\pm}_{\beta, \alpha}
	\ar@/_/[d]_{\pi_{\beta, \alpha}}
	\ar[r]^{\sigma} & 
	\Omega_{\beta, \alpha}^{\pm}
	\ar@/^/[d]^{\pi_{\beta, \alpha}} \\
	{J_{\beta, \alpha}}
	 \ar@/_/[u]_{\tau_{\beta,\alpha}^{\pm}}
	\ar[r]_{T^{\pm}_{\beta, \alpha}}  & 
	 \ar@/^/[u]^{\tau_{\beta,\alpha}^{\pm}}
	J_{\beta, \alpha}}
	\end{aligned}
	\end{align}
This result is readily verifiable from the definitions of the maps involved, see \cite{BHV}. From this, one may deduce, for $x \in [0, 1 + 1/(\beta-1)]$, that the words $\tau_{\beta,\alpha}^\pm(x - \alpha /(\beta - 1))$ are $\beta$-expansions of $x$.  It is worth noting that the expansion of a point $x$ given by $\tau_{\beta,0}^{+}(x)$, namely the greedy $\beta$-expansion of $x$, is lexicographically the largest $\beta$-expansion of $x$, and the expansion given by $\tau_{\beta,2-\beta}^-(x - (2 - \beta) /(\beta - 1))$, namely the lazy $\beta$-expansion of $x$, is lexicographically the smallest $\beta$-expansion of $x$, see \cite{MR1078082}.  Further, for Lebesgue all most all $x$, the expansion $\tau_{\beta,\alpha}^\pm(x - \alpha /(\beta - 1))$ lie in between the greedy and the lazy $\beta$-expansions of $x$, with respect to the lexicographic ordering, see \cite{DK:2002b}. There also exist $\beta$ such that the only $\beta$-expansion of one is the greedy $\beta$-expansion, such $\beta$ are called \textsl{univoque}, see \cite{Komornik:1998} for further details.

\begin{exm}\label{ex:inbetween}
For $\beta = (1 + \sqrt{5})/2$ and $x = 1$, we have the following.
	\begin{alignat*}{3}
	&\tau^{+}_{\beta, 0}(x) = (1, 1, \overline{0})							&\qquad	&\text{greedy golden mean expansion of} \; 1\\
	&\tau^{\pm}_{\beta, 1 - \beta/2}(x - (1-\beta/2)/(\beta-1)) = (\overline{1, 0})		&		&\text{symmetric golden mean expansion of} \; 1\\
	&\tau^{-}_{\beta, 2 - \beta}(x - (2-\beta)/(\beta-1)) = (0,\overline{1})			&		&\text{lazy golden mean expansion of} \; 1
	\end{alignat*}
For $\beta$ the largest positive real root of $z^{14} - 2 z^{13} + z^{11} - z^{10} - z^{7} + z^{6} - z^{4} + z^{3} - z + 1$, and $x = 1$,
	\begin{align*}
	\tau^{\pm}_{\beta, \alpha}(x - \alpha/(\beta-1)) = (1,1,1,0,0,1,0,1,1, \overline{1,0,0,1,0,1,0}),
	\end{align*}
for all $\alpha \in [0, 2-\beta]$. In \cite{MR2299792}, it was shown, in this latter case, that $\beta$ is the smallest univoque Pisot number.
\end{exm}
 
Next, we recall a result which shows that $\Omega_{\beta,\alpha}^\pm$ is completely determined its kneading invariants.

\begin{theorem} \cite{BHV,H,HS} \label{omegastructure}
Letting $\prec, \preceq, \succ, \succeq$ denote the lexicographic orderings on $\{0,1\}^\bN$, we have that
\begin{align*}
\Omega_{\beta,\alpha}^+ &= \left\{ \omega \in \{0,1\}^\bN \colon \text{for all} \; n \in \bN_0, \; \sigma^n(\omega) \prec \tau_{\beta, \alpha}^-(p) \; \text{or} \; \tau_{\beta, \alpha}^+(p) \preceq \sigma^n(\omega) \right\},\\
\Omega_{\beta,\alpha}^- &= \left\{ \omega \in \{0,1\}^\bN \colon \text{for all} \; n \in \bN_0, \; \sigma^n(\omega) \preceq \tau_{\beta, \alpha}^-(p) \; \text{or} \; \tau_{\beta, \alpha}^+(p) \prec \sigma^n(\omega) \right\}.
\end{align*}
\end{theorem}

A necessary and sufficient condition on the kneading invariants of an intermediate $\beta$-shift for determining when it is a subshift of finite type is as follows.

\begin{theorem}[{\cite{MR0346134,LSS,MR0142719}}]\label{Parry:GL}
For $(\beta,  \alpha) \in \Delta$, the intermediate shift $\Omega_{\beta,\alpha}$ is a subshift of finite type if and only if $\sigma(\tau_{\beta,\alpha}^\pm(p))$ are both periodic.
\end{theorem}

With the above at hand, it is natural to ask if $\beta \in  (1, 2)$ and $\alpha \in (0, 2-\beta)$, then is true that $\tau_{\beta,\alpha}^+(p)$ is periodic if and only if $\tau_{\beta,\alpha}^-(p)$ is periodic and vice versa?  In \Cref{+iff-} we show that this is indeed the case when $\beta$ is a multinacci number.  However, there exist values of $\beta \in (1, 2)$ for which this does not hold, as the following counterexample demonstrates.   Thus, it would be interesting to investigate if there exists other values of $\beta \in (1, 2)$, for which $\tau_{\beta,\alpha}^+(p)$ is periodic if and only if $\tau_{\beta,\alpha}^-(p)$.  In fact this idea is very closely linked to the concept of matching which has recently attracted much attention.

\begin{exm}\label{ex:+periodic-non-periodic}
Letting $\beta = \sqrt{\beta_{2}}$ and $\alpha = 2 - \beta_{2}$, we have that  $\tau^{+}_{\beta, \alpha}(p) = (\overline{1, 0, 0, 1})$ and $\tau^{-}_{\beta, \alpha}(p) = ( 0, 1, \overline{ 1, 0})$.  Recall, $\beta_{2}$ denotes the second multinacci number, namely the golden mean.
\end{exm}

Kalle and Steiner \cite{KalleSteiner} developed an analogous result to \Cref{Parry:GL} for determining when a $\beta$-shift is sofic; this allows us to conclude \Cref{cor:sofic} from \Cref{thm:thmB}.  Their result states the following.

\begin{theorem}[\cite{KalleSteiner}] \label{thm:sofic}
The subshift $\Omega_{\beta, \alpha}$ is sofic if and only if $\tau_{\beta, \alpha}^{\pm}(p)$ are both pre-periodic.
\end{theorem}

Our next proposition (\Cref{+iff-}) plays a key r\^ole in the proof of \Cref{SFTdense}.  We note that after the writing of this paper we became aware of \cite{KalleBruinCarminati} in which a proof of this result also appears.  However, for completeness we include a short justification for which we require an auxiliary lemma (\Cref{beginword}).

\begin{prop} \label{+iff-}
Fix an integer $m \geq 2$ and let $\alpha \in \Delta(\beta_m)\backslash \{0, 2-\beta_n\}$.  The kneading invariant $\tau_{\beta_m, \alpha}^-(p)$ is periodic if and only if the kneading invariant $\tau_{\beta_m, \alpha}^+(p)$ is periodic.
\end{prop}

\begin{lemma} \label{beginword}
Under the assumptions of \Cref{+iff-}, we have 
\begin{align*}
\tau_{\beta_m, \alpha}^+(p)\vert_{m+1} = (1 \underbrace{0, 0, \dots, 0, 0}_{ m-\text{times}})
\quad \text{and} \quad
\tau_{\beta_m, \alpha}^-(p)\vert_{m+1} = (0 \underbrace{1, 1, \dots, 1, 1}_{ m-\text{times}}).
\end{align*}
\end{lemma}

\begin{proof}
We present the proof for $\tau_{\beta_{m}, \alpha}^{-} (p)$; the proof for $\tau_{\beta_{m}, \alpha}^{+} (p)$ follows analogously.  From \Cref{criticalitinerarybegin} we know $\tau_{\beta, \alpha}^-(p)\vert_{2} = (0, 1)$, 
and, since $\beta_{m} \geq \beta_{2}$, by definition $(T_{\beta_{m}, \alpha}^{-})^{2}(p) = \beta_{m} + \alpha -1 > p$.  Suppose, for some $j \in \{ 1, 2, \dots m -1 \}$, that
\begin{align*}
\tau_{\beta_m, \alpha}^-(p)\vert_{j+1} = (0 \underbrace{1, 1, \dots, 1, 1}_{ j-\text{times}}).
\end{align*}
Let $S_{0}(x) \coloneqq \beta_{m} x + \alpha$ and $S_{1}(x) \coloneqq \beta_{m} x + \alpha - 1$. It suffices to show $\beta (T^{-}_{\beta, \alpha})^{j+1}(p) + \alpha =\beta({S_{1}}^{j} \circ S_{0}(p)) + \alpha$ is strictly greater than $1$. To this end, observe that
\begin{align*}
\beta(S_{1}^{j} \circ S_0(p)) + \alpha
&= \beta_m^{j+1} + \alpha(\beta_m^j + \beta_m^{j-1} + \dots +\beta_m + 1) -\beta_m^j - \beta_m^{j-1} - \dots -\beta_m \\
&> \beta_m^{j+1} -\beta_m^j - \beta_m^{j-1} - \dots -\beta_m \\
&\geq \beta_{j+1}^{j+1} -\beta_{j+1}^j - \beta_{j+1}^{j-1} - \dots -\beta_{j+1} = 1.
\end{align*}
The first line follows from an elementary induction argument and the definition of $S_{0}^{-}$ and $S_{1}^{-}$; the second line holds since $\alpha > 0$; the last and penultimate lines are a consequence of the facts $(\beta_{k})_{k \in \mathbb{N}}$ is an increasing sequence and $\beta_{j+1}$ is the unique real zero of the polynomial $x^{j+1} - x^{j} - \dots - x - 1$ in $(1, 2)$.
\end{proof}

\begin{proof}[Proof of \Cref{+iff-}]
By \Cref{beginword}, we have
\begin{align*}
\tau_{\beta_m, \alpha}^+(p)\vert_{m+1} = (1, \underbrace{0, 0, \dots, 0, 0}_{m-\text{times}})
\quad \text{and} \quad
\tau_{\beta_m, \alpha}^-(p)\vert_{m+1} = (0, \underbrace{1, 1, \dots, 1, 1}_{m-\text{times}}).
\end{align*}
Letting $S_{0}$ and $S_{1}$ be as in the proof of \Cref{beginword}, an elementary calculation yields the following.
	\begin{align*}
	(T_{\beta_{m}, \alpha}^{+})^{m+1}(p) &= S_{0}^{m} \circ S_{1}(p) 
	\,= \alpha (\beta_m^{m-1} + \beta_m^{m-2} + \dots + \beta_m + 1) 
	= \alpha \beta_m^m\\
	(T_{\beta_{m}, \alpha}^{-})^{m+1}(p) &= 
	\begin{aligned}[t]
	S_{1}^{m} \circ S_{0}(p) 
	&= \alpha (\beta_m^{m-1} + \beta_m^{m-2} + \dots + \beta_m + 1)\\
	&\hspace{1.125em} +  (\beta_m^m - \beta_m^{m-1} - \beta_m^{m-2} - \dots - \beta_m - 1) \\
	&= \alpha (\beta_m^{m-1} + \beta_m^{m-2} + \dots + \beta_m + 1)
= \alpha \beta_m^m
	\end{aligned}
	\end{align*}
Namely, $(T_{\beta_m, \alpha}^+)^{m+1}(p) = (T_{\beta_m, \alpha}^-)^{m+1}(p)$.  Thus, $\tau_{\beta_m, \alpha}^-(p)$ is periodic if and only if $\tau_{\beta_m, \alpha}^+(p)$ is periodic.
\end{proof}

\subsection{Uniform Lorenz maps}\label{sec_uniform}

A class of maps closely related to intermediate \mbox{$\beta$-transformations}, and which have been well studied, are Lorenz maps.  They are expanding interval maps with a single discontinuity. Here, we consider the sub-class of \textsl{uniform Lorenz maps} $U_{\beta, p}^{\pm} \colon [0, 1] \circlearrowleft$ defined, for $\beta \in (1,2)$ and $q \in [1 - 1/\beta, 1/\beta]$, by 
\begin{align*}
U_{\beta, q}^+(x) \coloneqq \begin{cases}
             \beta x   & \text{if} \; x < q,\\
             \beta x + 1 - \beta  & \text{if} \; x \geq q.
       \end{cases}
\quad \text{and} \quad      
U_{\beta, q}^-(x) \coloneqq \begin{cases}
              \beta x  & \text{if} \; x \leq q,\\
             \beta x + 1 - \beta  & \text{if} \; x > q,
       \end{cases}
\end{align*}
Let us now describe the relation between uniform Lorenz maps and $\beta$-transformations.  For this we require the following concept, which determines when two dynamical systems are `the  same'.  Let $X$ and $Y$ denote two topological spaces and let $f \colon X \circlearrowleft$ and $g\colon Y \circlearrowleft$. We say that $f$ and $g$ are \textsl{topologically conjugate} if there exists a homeomorphism $h \colon X \to Y$ such that $h \circ f = g \circ h$. The maps $f$ and $g$ are called \textsl{topologically semi-conjugate} if $h$ is a continuous surjection.

An elementary calculation shows that $T_{\beta, \alpha}^{\pm}$ and $U^{\pm}_{\beta, 1 + (\alpha-1)/\beta}$ are topologically conjugate, where the conjugating homeomorphism is given by $x \mapsto (\beta-1)(x + \alpha/(\beta - 1))$.

Similar to $\beta$-transformations, Lorenz maps have associated expansion maps $\mu_ {\beta, q}^{\pm} \colon [0, 1] \to \{ 0, 1\}^{\bN}$ defined by $\mu_{\beta, q}^\pm(x) \coloneqq (\nu_{1}^{\pm}(x), \nu_{2}^{\pm}(x), \dots)$, where, for $n \in \bN$,
\begin{align*}
\nu_{n}^{+}(x) \coloneqq 
\begin{cases}
0 & \text{if} \; (U_{\beta,q}^{+})^{n-1}(x) < q,\\
1 & \text{otherwise,} 
\end{cases}
\quad \text{and} \quad
\nu_{n}^{-}(x) = 
\begin{cases}
0 & \text{if} \; (U_{\beta,q}^{-})^{n-1}(x) \leq q, \\
1 & \text{otherwise,}
\end{cases}
\end{align*}
as well as an associated projection map $\rho_{\beta} \colon \{ 0, 1 \}^{\bN} \to [0, 1]$ given by $\rho_{\beta} (\omega_{1}, \omega_{2}, \dots) \coloneqq (\beta - 1) \sum_{i = 1}^{\infty} \omega_{i} \beta^{-i}$.  As in the setting of \Cref{sec:betashifts} we have that the following diagram commutes.
	\begin{align*}
	\begin{aligned}
	\xymatrix@C+2pc{
	 \mu_{\beta, q}^{\pm}([0, 1])
	\ar@/_/[d]_{\rho_{\beta}}
	\ar[r]^{\sigma} & 
	 \mu_{\beta, q}^{\pm}([0, 1])
	\ar@/^/[d]^{\rho_{\beta}} \\
	{[ 0, 1 ]}
	 \ar@/_/[u]_{\mu_{\beta,q}^{\pm}}
	\ar[r]_{U_{\beta, q}^{\pm}}  & 
	 \ar@/^/[u]^{\mu_{\beta,q}^{\pm}}
	{[ 0, 1 ]}}
	\end{aligned}
	\end{align*}
Additionally, we have the following monotonicity result.

\begin{prop} \label{tauinc} \cite{BHV,COOPERBAND202096}
Let $\beta \in (1, 2)$ be fixed.  The map $x \mapsto \mu_{\beta, x}^{+}(x)$ is right-continuous and strictly increasing. Similarly, $x \mapsto \mu_{\beta, x}^{-}(x)$ is left-continuous and strictly increasing.  Moreover, points of discontinuity, for both maps, only occur at periodic points.
\end{prop} 

The main benefit of using uniform Lorenz maps stems from the idea that every $\beta$-trnasformation has a  realisation as a uniform Lorenz map, as discussed above, and that every uniform Lorenz map is defined on $[0, 1]$ and has the same fixed points.  Thus, it allows one to easily compare the kneading invariants of systems with the same expansion rate, namely $\beta$, but with different translates, namely $\alpha$.

\section{Fiber Denseness of intermediate $\beta$-shifts of finite type\\
-- Proof of \Cref{SFTdense} --}\label{sec:Proof_Thm_1}

The aim of this section is to prove \Cref{SFTdense}.  We divide the proof into two parts.  We show that the sets $\P^{\pm}(\beta) \coloneqq \{ \alpha \in \Delta(\beta) \colon \tau_{\beta, \alpha}^{\pm}(p) \; \text{is periodic} \}$, for a given $\beta \in (1, 2)$, are dense in $\Delta(\beta)$ with respect to the Euclidean norm, and with the help of \Cref{+iff-}, we have that $\tau_{\beta_{k}, \alpha}^{+}$ is periodic if and only if $\tau_{\beta_{k}, \alpha}^{-}$ is periodic.  \Cref{SFTdense} follows by combining these two results together with \Cref{Parry:GL}.

\begin{proof}[Proof of \Cref{SFTdense}]

Fix $(\beta, \alpha) \in \Delta$ with $\alpha \not\in \{ 0, 2-\beta\}$.  Let $q = 1 + (\alpha - 1)/\beta$, so that $U_{\beta, q}^\pm$ is topologically conjugate to $T_{\beta,\alpha}^\pm$.  It is sufficient to show that there exists $q_{s}^{\pm}$ sufficiently close to $q$ in $((\beta-1)/\beta,  1/\beta)$ with $\mu_{\beta, q_{s}^{\pm}}^{\pm} (q_{s}^{\pm})$ periodic. We present the proof for $\mu_{\beta, q_{s}^{-}}^{-} (q_{s}^{-})$; the proof for $\mu_{\beta, q_{s}^{+}}^{+} (q_{s}^{+})$ follows analogously.  To this end, suppose $\mu_{\beta, q}^{-} (q)$ is not periodic, otherwise set $q_{s}^{-} = q$.  Fix $k \in \bN$ and set
\begin{align*}
\delta_k^{(q)} \coloneqq \min \{ \beta^{-k} \lvert (U_{\beta, q}^-)^{l}(q) - q \rvert \colon l \in \{ 1, 2, \dots, k \} \}.
\end{align*}
Observe $\mu_{\beta, q'}^{-}(q')\vert_{k+1} = \mu_{\beta, q}^{-}(q)\vert_{k+1}$, for all $q' \in [q - \delta_{k}^{(q)}, q)$. Let $j >k$ be the maximal integer such that
	\begin{align*}
	 \mu_{\beta, q - \delta_k^{(q)}}^-(q - \delta_k^{(q)})\vert_j = \mu_{\beta, q}^-(q)\vert_j,
	 \quad
	(\mu_{\beta, q - \delta_k^{(q)}}^-(q - \delta_k^{(q)}))_{j+1} = 0 
 	\quad \text{and} \quad
 	(\mu_{\beta,  q}^-(q))_{j+1} = 1.
 	\end{align*}
The existence of $j$ is given by \Cref{tauinc}. Let $A \subseteq [q - \delta_k, q)$ be defined by
 	\begin{align*}
	 A \coloneqq \{ x \in [q - \delta_k^{(q)}, q) \colon  \mu_{\beta,x}^-(x)\vert_j = \mu_{\beta,q}^-(q)\vert_j \; \text{and} \; (\mu_{\beta,x}^-(x))_{j+1} = 0 \}.
	\end{align*}
\Cref{tauinc} ensures that $A$ is a non-empty, connected and closed, in particular that $q_{s}^{-} \coloneqq \sup(A) \in A$.  By way of contradiction, suppose that $\mu_{\beta, q_{s}^{-}}^{-} (q_{s}^{-})$ is not periodic, in which case,
	\begin{align*}
	(U_{\beta, q_s^{-}}^{-})^j(q_{s}^{-}) < q_{s}^{-}.
 	\end{align*} 
By \Cref{tauinc}, there exists $\delta_{j+1}^{(q_{s})} > 0$ so that, if $q' \in (q_s, q_s+ \delta_{j+1}^{(q_s)})$, then $\mu_{\beta, q_{s}^{-}}^-(q_{s}^{-})\vert_{j+1} = \mu_{\beta,q'}^-(q')\vert_{j+1}$.  This implies $q' \in A$; contradicting the fact $q_s^{-}$ is the supremum of $A$.  Therefore, since $\delta_{k}^{(q)} < \beta^{-k}$, given $q \in (1 - 1/\beta, 1/\beta)$ and $\epsilon > 0$, there exist $q_{s}^{-} \in (1 - 1/\beta, q)$ and $k \in \bN$ with $q - q^{-}_{s} \leq \delta_{k}^{(q)} < \epsilon$ and $\mu_{\beta, q^{-}_s}^{\pm} (q^{-}_s)$ periodic.

\Cref{+iff-} implies that $\P^{+}(\beta_{n}) = \P^{-}(\beta_{n})$, for all integers $n \geq 2$.  Thus, we have that the set $\{ \alpha \in \Delta(\beta) \colon \tau_{\beta, \alpha}^{+}(p) \; \text{and} \; \tau_{\beta, \alpha}^{-}(p) \; \text{are periodic} \}$ is dense in $\Delta(\beta)$ with respect to the Euclidean norm.  With this at hand, an application of \Cref{Parry:GL} completes the proof.
\end{proof}

\section{Periodic expansions of Pisot and Salem numbers\\
-- Proof of \Cref{thm:thmB} -- }\label{sec:Proof_Thm_2_&_3}

Throughout this section, let $\beta \in (1,2)$ denote an algebraic integer with minimal polynomial 
	\begin{align*}
	P(z) \coloneqq \sum_{i = 0}^{d-1} a_{i} z^{i} + z^{d},
	\end{align*}
where $z \in \bC$, $d \in \bN$ and $a_1, a_{2}, \dots, a_{d} \in \bZ$.  In which case, $x \in \bQ(\beta) \cap J_{\beta, \alpha}$ can be written in the form 
	\begin{align} \label{eqn:algebraicexp}
	x = q^{-1} \sum_{i = 0}^{d-1} p_i \beta^i,
	\end{align}
where $p_1, p_{2}, \dots, p_{d-1} \in \bZ$ and $q \in \bN$.  We assume that the integer $q$ in \eqref{eqn:algebraicexp} is as small as possible yielding a unique representation for $x$.  Let $\widehat{p}_{1}, \widehat{p}_{2}, \dots, \widehat{p}_{d-1} \in \bZ$ and and $\widehat{q} \in \bN$ denote the corresponding terms for $\alpha \in \bQ(\beta)$:
	\begin{align} \label{eqn:algebraicexp2}
	\alpha = \widehat{q}^{\,-1} \sum_{i = 0}^{d-1} \widehat{p}_i \beta^i.
	\end{align}
Fix $\alpha \in \bQ(\beta) \cap (0, 2-\beta)$ and $x \in \bQ(\beta) \cap J_{\beta, \alpha}$ with the forms given in \eqref{eqn:algebraicexp} and \eqref{eqn:algebraicexp2}. For $i \in \bN$, let $\omega_{i}^{\pm}(x)$ respectively denote the $i$-th letter of $\tau_{\beta, \alpha}^\pm (x)$.  From the commutative diagram given in \eqref{diag:commutative2}, we have, for $n$ a non-negative integer, that
\begin{align} \label{eqn:rho}
\rho^{(n,\pm)}(x) \coloneqq \beta^n\left(x - \sum_{i=1}^{n}\omega_i^\pm(x) \beta^{-i} + \alpha \sum_{i = 1}^{n} \beta^{-i}\right) = (T_{\beta, \alpha}^\pm)^n(x).
\end{align}

\begin{lemma} \label{lemma1}
For $x \in \bQ(\beta) \cap J_{\beta, \alpha}$ and $n \in \bN_{0}$, there exists a unique vector $(r_1^{(n, \pm)}(x), \dots, r_d^{(n, \pm)}(x))$ in $\bZ^d$ with 
	\begin{align} \label{eqn:rhoiterate}
	\rho^{(n, \pm)}(x) = (\widehat{q}q)^{-1} \sum_{i=1}^{d} r_i^{(n, \pm)}(x) \beta^{-i}.
	\end{align}
\textnormal{For ease of notation, and when the dependency on the point $x$ is clear, we write $\mathbf{r}^{(n, \pm)} = (r_1^{(n, \pm)}, \dots, r_d^{(n, \pm)})$ in replace of $\mathbf{r}^{(n, \pm)}(x) = (r_1^{(n, \pm)}(x), \dots, r_d^{(n, \pm)}(x))$.}
\end{lemma}

\begin{proof}
By \eqref{eqn:algebraicexp}, \eqref{eqn:algebraicexp2} and \eqref{eqn:rho} we have that
	\begin{align*}
	\rho^{(1,\pm)}(x)
	= q^{-1} \sum_{i = 1}^{d} p_{i - 1} \beta^{i} - \omega^{\pm}_{1}(x) + \alpha
	=  (\widehat{q}q)^{-1} \left( \widehat{q} \sum_{i = 1}^{d} p_{i - 1} \beta^{i} - \widehat{q}q\omega^{\pm}_{1}(x) + q\sum_{i = 0}^{d-1} \widehat{p}_i \beta^i \right).
	\end{align*}
The result for $n = 1$ follows from the fact that $q$ and $\widehat{q}$ are fixed and that $B \coloneqq \{\beta, \beta^2, \dots, \beta^{d}\}$ is a basis for $\bQ(\beta)$.  An inductive argument yields the general result.
\end{proof} 

\begin{lemma} \label{lemma2}
For $x \in \bQ(\beta) \cap J_{\beta, \alpha}$, $n \in \bN_{0}$ and $\gamma$ a Galois conjugate of $\beta$,
\begin{align} \label{eqn:gamma}
\gamma^n\left(q^{-1} \sum_{i = 0}^{d-1} p_i \gamma^i - \sum_{i=1}^{n}\omega_i^{\pm}(x) \gamma^{-i} + \widehat{\alpha} \sum_{i = 1}^{n} \gamma^{-i} \right)  = (\widehat{q}q)^{-1} \sum_{i=1}^{d} r_i^{(n, \pm)} \gamma^{-i},
\end{align}
where $\widehat{\alpha} = \widehat{q}^{\,-1} \sum_{i = 0}^{d-1} \widehat{p}_i \gamma^i$.  Moreover, if $\lvert \gamma \rvert >1$ and if $x \in \Per^{\pm}(\beta,\alpha) \cap \bQ(\beta)$, then
\begin{align}\label{eqn:pigamma}
q^{-1} \sum_{i = 0}^{d-1} p_i \gamma^i = \frac{\widehat{\alpha}}{1 - \gamma} + \sum_{i = 1}^{\infty} \omega_i^{\pm}(x) \gamma^{-i}.
\end{align}
\end{lemma}

\begin{proof}
Combining \eqref{eqn:algebraicexp}, \eqref{eqn:rho} and \eqref{eqn:rhoiterate}, we obtain that $\beta$ satisfies the polynomial equation
\begin{align} \label{eqn:gamma1}
z^{n+d}\left( \widehat{q} \sum_{i = 0}^{d-1} p_i z^i - \widehat{q}q \sum_{i=1}^{n}\omega_i^{\pm}(x) z^{-i} + q \left(\sum_{j = 0}^{d-1} \widehat{p}_j z^i \right) \left( \sum_{i = 1}^{n} z^{-i} \right)\right)  = \sum_{i=1}^{d} r_i^{(n, \pm)} z^{d-i}
\end{align}
for all $n \geq 0$. Since $\gamma$ is a Galois conjugate of $\beta$, it is also a solution to \eqref{eqn:gamma1}, which proves \eqref{eqn:gamma}. If $\lvert \gamma \rvert >1$ and if $x \in \Per^{\pm}(\beta, \alpha)$, then the cardinality of the set $\{ \mathbf{r}^{(n, \pm)} \colon n \in \mathbb{N}_{0} \}$ is finite, and thus 
\begin{align} \label{eqn:boundednorm}
c^{\pm} \coloneqq \sup \{ \max \{ \lvert r_k^{(n, \pm)} \rvert \colon k \in \{ 1, \dots, d \} \} \colon  n \in \bN  \} < \infty.
\end{align}
Combining this with \eqref{eqn:gamma} we obtain
\begin{align*}
\left\lvert q^{-1} \sum_{i=0}^{d-1} p_i \gamma^{i} - \sum_{i=1}^{n}\omega_i^{\pm}(x) \gamma^{-i} + \widehat{\alpha} \sum_{i = 1}^{n} \gamma^{-i} \right\rvert \leq (\widehat{q} q)^{-1} c^{\pm} d \lvert \gamma \rvert^{-n}. 
\end{align*}
Letting $n$ tend to infinity in the above equation yields \eqref{eqn:pigamma}.
\end{proof}

With the above two lemmas at hand we are ready to prove \Cref{thm:thmB}\,\ref{thm:thmB1}. 

\begin{proof}[Proof of \Cref{thm:thmB}\,\ref{thm:thmB1}]
We show the result for $T_{\beta, \alpha}^+$ noting that the proof is analogous for $T_{\beta,\alpha}^-$.  By way of contradiction, suppose there exists a Galois conjugate $\gamma \neq \beta$ of $\beta$ with $\lvert \gamma \rvert > 1$. Let $x \in [\alpha,\beta + \alpha -1]$ and let $a, b \in J_{\beta, \alpha}$ be such that $a < b$ and $T_{\beta, \alpha}^{+}(a) = T_{\beta, \alpha}^{+}(b) = x$. Set $\delta \coloneqq \lvert \beta^{-1} - \gamma^{-1} \rvert$ and let $\eta \coloneqq \max\{\beta^{-1}, \lvert \gamma \rvert^{-1}\}$. Choose $m \in \bN$ with $\eta^{m+1}/(1 - \eta) < \delta/2$.

Let $a',b' \in \bQ \cap J_{\beta, \alpha}$ with $\tau_{\beta,\alpha}^+(a)\vert_m = \tau_{\beta,\alpha}^+(a')\vert_m$ and $\tau_{\beta,\alpha}^+(b)\vert_m=\tau_{\beta,\alpha}^+(b')\vert_m$; the existence of $a'$ and $b'$ is guaranteed by \Cref{tauinc}.  By \eqref{diag:commutative2} and how $a'$ and $b'$ have been chosen, $(\omega^{+}_2(a'), ..., \omega^{+}_m(a')) = (\omega^{+}_2(b'), ..., \omega^{+}_m(b'))$, $\omega^{+}_1(a') = 0$ and  $\omega^{+}_1(b') = 1$.  An application of \Cref{lemma2} in tandem with \eqref{eqn:algebraicexp}, our hypothesis and the fact that $\gamma$ is a Galois conjugate of $\beta$, yields the  following.
	\begin{align*}
	a' &= \frac{\alpha}{1 - \beta} + \sum_{i=1}^{\infty} \omega^{+}_i(a') \beta^{-i} = \frac{\widehat{\alpha}}{1 - \gamma} + \sum_{i=1}^{\infty} \omega^{+}_i(a') \gamma^{-i}\\
	b' &= \frac{\alpha}{1 - \beta} + \sum_{i=1}^{\infty} \omega^{+}_i(b') \beta^{-i} = \frac{\widehat{\alpha}}{1 - \gamma} + \sum_{i=1}^{\infty} \omega^{+}_i(b') \gamma^{-i}
	\end{align*}
From this we obtain the following chain of inequalities.
\begin{align*}
\begin{aligned}
\delta = \lvert\beta^{-1} - \gamma^{-1}\rvert
&= \left\lvert\frac{\alpha}{1 - \beta} + \sum_{i=2}^{\infty} \omega^{+}_i(b') \beta^{-i} -  \frac{\widehat{\alpha}}{1 - \gamma} - \sum_{i=2}^{\infty}\omega^{+}_i(b') \gamma^{-i}\right\rvert \\
&\leq \left\lvert\frac{\alpha}{1 - \beta} + \sum_{i=2}^{\infty} \omega^{+}_i(b') \beta^{-i} -a'\right\rvert + \left\lvert a' -  \frac{\widehat{\alpha}}{1 - \gamma} - \sum_{i=2}^{\infty}\omega^{+}_i(b') \gamma^{-i}\right\rvert \\
&\leq \left\lvert \sum_{i=2}^{\infty} \omega^{+}_i(b') \beta^{-i} - \sum_{i=1}^{\infty} \omega^{+}_i(a') \beta^{-i}\right\rvert + \left\lvert \sum_{i=1}^{\infty} \omega^{+}_i(a') \gamma^{-i} - \sum_{i=2}^{\infty}\omega^{+}_i(b') \gamma^{-i}\right\rvert \\
&\leq \sum_{i = m+1}^{\infty} \lvert\omega^{+}_i(b') - \omega^{+}_i(a')\rvert\beta^{-i} + \sum_{i = m+1}^{\infty} \lvert\omega^{+}_i(b') - \omega^{+}_i(a')\rvert \lvert \gamma\rvert^{-i}
\leq 2 \eta^{m+1}(1 - \eta)^{-1} < \delta
\end{aligned}
\end{align*}
This yields a contradiction, and concludes the proof.
\end{proof}

For the proof of \Cref{thm:thmB}\,\ref{thm:thmB2} we require an additional lemma.

\begin{lemma} \label{lemma:tfae}
Set $\beta = \gamma_1$ and let $\gamma_2, \dots, \gamma_d$ denote the Galois conjugates of $\beta$.  For $x \in \bQ(\beta) \cap J_{\beta, \alpha}$, $n \in \bN_{0}$ and $ i \in \{ 1,2,\dots,d\}$ set
	\begin{align} \label{eqn:rho_i}
	\rho_i^{(n, \pm)}(x) \coloneqq q^{-1} \sum_{k=1}^{d} r_k^{(n, \pm)}(x) \gamma_i^{-k}.
	\end{align}
The following are equivalent.
	\begin{enumerate}[label={(\roman{enumi})}]
	\item \label{eqn:tfae1} $x \in \Per^{\pm}(\beta,\alpha)$
	\item \label{eqn:tfae2} $\displaystyle \max \{ \sup \{ \lvert\rho_i^{(n, \pm)}(x)\rvert \colon n \in \bN_{0} \} \colon i \in  \{ 1, \dots, d \}\} < \infty$
	\item \label{eqn:tfae3} $\displaystyle \sup \{ \max \{ \lvert r_k^{(n, \pm)}(x) \rvert \colon k \in \{1, \dots, d \} \} \colon n \in \bN_{0} \}  < \infty$
	\end{enumerate}
\end{lemma}

\begin{proof}
A similar argument to that given in the proof of \Cref{lemma2}, where we obtained \eqref{eqn:boundednorm}, shows \ref{eqn:tfae1} implies \ref{eqn:tfae3}.  That \ref{eqn:tfae3} implies \ref{eqn:tfae2} follows from \eqref{eqn:rho_i}.  To complete the proof we show \ref{eqn:tfae2} implies \ref{eqn:tfae1}.  To this end, assume \ref{eqn:tfae2} and set
	\begin{align} \label{eqn:matrix}
	\mathbf{v}^{(n, \pm)}(x) \coloneqq q 
		\begin{pmatrix}
		\rho_1^{(n, \pm)}(x) \\ \vdots \\ \rho_d^{(n, \pm)}(x)
		\end{pmatrix}
		=
		\underbrace{\begin{pmatrix}
		\gamma_1^{-1} & \gamma_1^{-2} & \dots & \gamma_1^{-d} \\
		\vdots & \vdots & \ddots & \vdots \\
		\gamma_d^{-1} & \gamma_d^{-2} & \dots & \gamma_d^{-d} \\
		\end{pmatrix}}_{\displaystyle\eqqcolon M_{\beta}}
		\begin{pmatrix}
		r_1^{(n, \pm)}(x) \\ \vdots \\ r_d^{(n, \pm)}(x)
		\end{pmatrix}.
	\end{align}
By assumption, there exists $c^{\pm} \in \bR$ with $\lVert \mathbf{v}^{(n, \pm)}(x) \rVert \leq c^{\pm}$, for all $n \in \bN_{0}$.  Since the Galois group of a finite Galois extension acts transitively on the roots of any minimal  polynomial, $M_{\beta}$ is a non-singular matrix.  This implies there exists $k^{\pm} \in \bZ$ with $\lVert \mathbf{r}^{(n, \pm)}(x) \rVert \leq k^{\pm}$, for all $n \in \bN_{0}$. Hence, as $\mathbf{r}^{(n, \pm)}(x) \in \bZ^{d}$ it follows that $\mathbf{r}^{(m, \pm)}(x) = \mathbf{r}^{(n, \pm)}(x)$, and therefore $\rho^{(m, \pm)}_{i} (x) = \rho^{(n, \pm)}_{i}(x)$, for some $m, n \in \bN_{0}$ with $m \neq n$ and all $i \in \{ 1, \dots, d \}$.  An application of \Cref{lemma1} and \eqref{eqn:rho} yields the required result.
\end{proof}

\begin{proof}[Proof of \Cref{thm:thmB}\,\ref{thm:thmB2}]
Fix $x \in \bQ(\beta) \cap [0,1]$ with the form given in \eqref{eqn:algebraicexp}.  As in \Cref{lemma:tfae}, set $\gamma_{1} = \beta$ and let $\gamma_2, \dots, \gamma_d$ denote the Galois conjugates of $\beta$.  Since, by assumption, $\beta$ is a Pisot number, it follows $\displaystyle{\eta \coloneqq \max \{ \lvert \gamma_j \rvert \colon j \in \{ 2, \dots, d \} \} <1}$.  For $j \in \{2, 3, \dots, d\}$, let
	\begin{align*}
	\widehat{\alpha}_{j} \coloneqq \widehat{q}^{\,-1} \sum_{i = 0}^{d-1} \widehat{p}_i {\gamma_{j}}\!^i,
	\end{align*}
and set $\widetilde{\alpha} \coloneqq \max \{ \lvert \widehat{\alpha}_{j} \rvert \colon j\in \{2, 3, \dots, d\} \}$.  By \eqref{eqn:gamma1} and \eqref{eqn:rho_i} we have
	\begin{align*}
	\lvert\rho_i^{(n, \pm)}(x)\rvert \leq q^{-1} \sum_{j = 0}^{d-1}\lvert p_j \rvert \eta^{n + j} + \sum_{i = 0}^{n-1} (1 + \widetilde{\alpha})\eta^{n + i}
	\end{align*}
for all $n \in \bN_{0}$ and $i \in \{ 2, \dots, d \}$. This in combination with \eqref{eqn:rho} yields that \Cref{lemma:tfae}\,\ref{eqn:tfae2} is satisfied, and thus $x \in \Per^{\pm}(\beta,\alpha)$.
\end{proof}

\section{Periodic expansions of Pisot and Salem numbers\\
--  Proof of \Cref{cor:sofic,thm:thmC} --}\label{sec:transitivity}

The aim of this final section is to provide an overview of the results of \cite{G:1990,Par:1979,MR0166332,Parry:1979} which in combination  with our results (\Cref{SFTdense,thm:thmB}) yield \Cref{thm:thmC,thm:thmD}.

An interval map $T  \colon [a, b]  \circlearrowleft$ is called \textsl{topologically transitive} if for all open subintervals $J$ there exists $m \in \bN$ with 
	\begin{align*}
	\bigcup_{k = 0}^{m} T^{k}(J) \supseteq (a, b).
	\end{align*}
For $\beta \in  (1, 2)$, Parry \cite{Parry:1979} showed $T_{\beta,1 - \beta/2}^{\pm}$ is topologically transitive if and only if $\beta > \sqrt{2}$.  This result was later generalised by Palmer \cite{Par:1979} and Glendinning \cite{G:1990} who classified the set of points $(\beta, \alpha) \in \Delta$ with $T_{\beta,\alpha}^{\pm}$ is topologically transitive.

In order to state the results of Parry, Palmer and Glendinning we require the following.  Let $n, k \in \mathbb{N}$ with $1 \leq k < n$ and $\operatorname{gcd}(n, k) = 1$, and let $I_{n, k}(\beta)$ be  as in \eqref{eq:Ink}.   Define $D_{n, k}$ to be the set 
	\begin{align*}
	\{ (\beta, \alpha) \in \Delta \colon \beta \in (1, 2^{1/n}] \; \text{and} \; \alpha \in I_{n, k}(\beta) \}, 
	\end{align*}
see \Cref{Fig:ParPlot} for an illustration of the intervals $I_{n, k}$ and the regions $D_{k, n}$.

\begin{figure}[t]
 \includegraphics[width=0.64\textwidth]{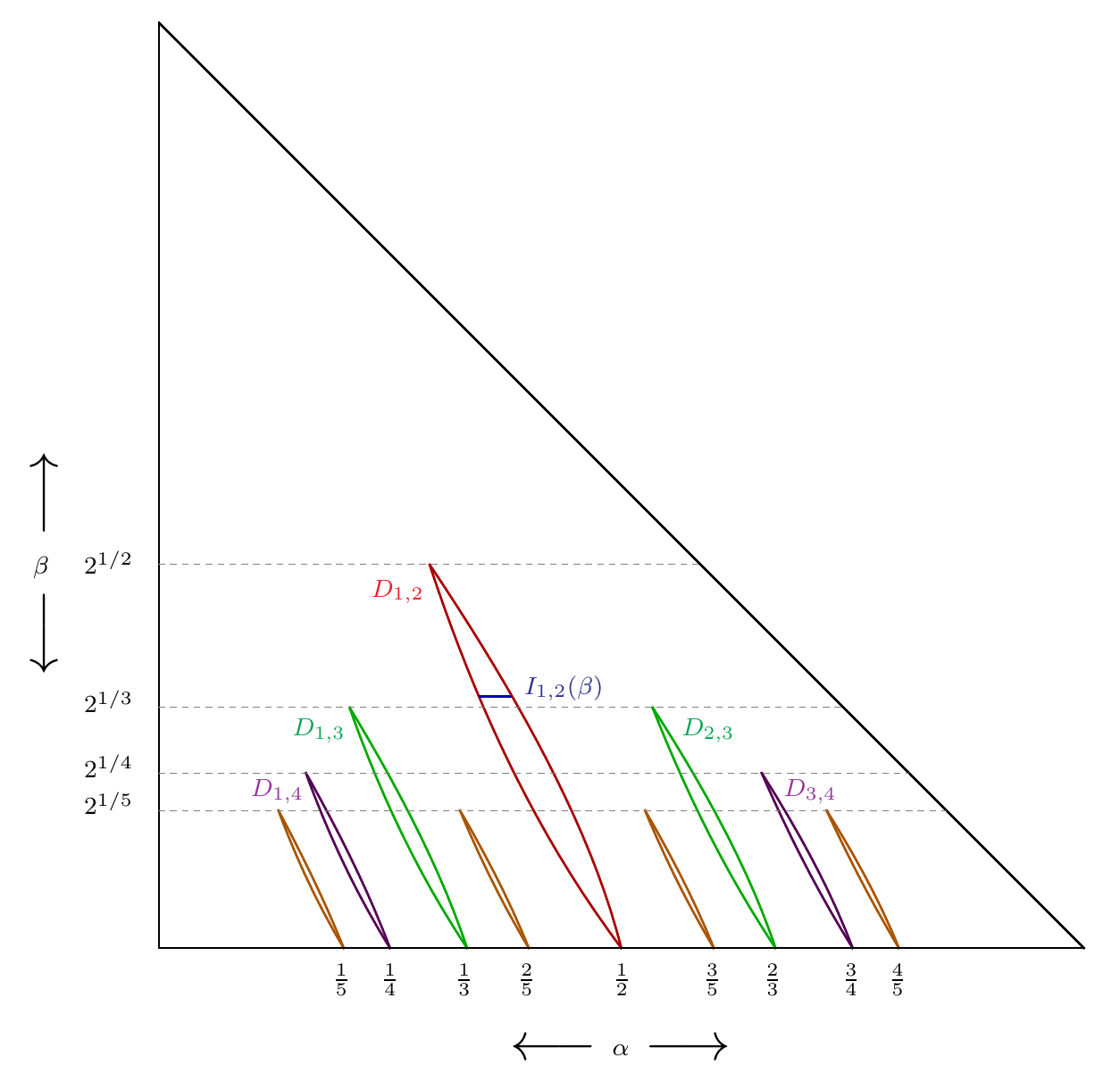}
\vspace{-1em}
	\caption{Plot of the parameter space $\Delta$, together with boundary of the regions $\color{Red}D_{1,2}$, $\color{Green}D_{1,3}$, $\color{Green}D_{2,3}$, $\color{Purple}D_{1,4}$, $\color{Purple}D_{3,4}$, $\color{Orange}D_{1,5}$, $\color{Orange}D_{2,5}$, $\color{Orange}D_{3,5}$, $\color{Orange}D_{4,5}$, $\color{Orange}D_{1,6}$ and $\color{Orange}D_{5,6}$.  Further, in blue, a sketch of the interval $I_{1,2}(\beta)$, where $\beta$ is the square root of the golden mean, is given.}
\label{Fig:ParPlot}
\end{figure}

\begin{theorem}[{\cite{G:1990,Par:1979,Parry:1979}}]\label{thm:Palmer}
Let $(\beta,\alpha) \in \Delta$.  The tuple $(\beta,\alpha) \in D_{k, n}$, for some $n, k \in \bN$ with $1\leq k < n$ and $\mathrm{gcd}(k, n) = 1$, if and only if $T_{\beta, \alpha}^{\pm}$ is not topologically transitive.
\end{theorem}

A main ingredient in the proof of this result is to show that for given $n, k \in \mathbb{N}$ with $1 \leq k < n$ and $\operatorname{gcd}(n, k) = 1$, there exists a one-to-one correspondence between points in $\Delta$ and points in $D_{n, k}$.  More precisely, on the one hand, given $(\beta, \alpha) \in \Delta$, there exists a unique $a \in I_{n, k}(\sqrt[n]{\beta})$, namely $a = \alpha_{n, k}(\beta, \alpha)$, see \Cref{thm:thmD}, such that,  $T^{\pm}_{\beta, \alpha}\vert_{[0, 1]}$ and  $(T^{\pm}_{\sqrt[n]{\beta}, a})^{n}\vert_{[a, \sqrt[n]{\beta} + a - 1]}$ are topologically conjugate with conjugating map $\Phi(x) \coloneqq (\sqrt[n]{\beta} - 1)x + a$; on the other hand, given $(\beta, \alpha) \in D_{n, k}$, there exists $a \in [0, 2 - \beta^{n}]$, namely 
	\begin{align*}
	a = \begin{cases}
	\displaystyle 1 - \frac{-\alpha(\beta^{n} - 1) + \beta - 1 }{(\beta - 1)(1 - \beta^{-1})} & \text{if} \; k = 1,\\[1em]
	\displaystyle 1 - \frac{-\alpha(\beta^{n} - 1) + (\beta - 1) \sum_{j = 1}^{s} W_{j} }{(\beta - 1)(1 - \beta^{-1})} & \text{otherwise}. 
	\end{cases}
	\end{align*}
such that $(T^{\pm}_{\beta, \alpha})^{n}\vert_{[\alpha, \beta+\alpha-1]}$ and  $T^{\pm}_{\beta^{n}, a}$ are topologically conjugate, where the conjugating map is given by $x \mapsto (\beta - 1)^{-1}(x - \alpha)$ and where $s \in \{ 0, 1, \dots, k-1\}$ satisfies $n = s \bmod{k}$ and $W_{j}$ is as defined in \eqref{eq:wj}.  Moreover, in the case that $\beta \neq 2^{1/n}$ and $(\beta, \alpha) \in D_{n, k}$ 
	\begin{align}\label{eq:intersection_transitive_empty}
	\overline{(T^{\pm}_{\beta, \alpha})^{i}([0,1])} \cap \overline{(T^{\pm}_{\beta, \alpha})^{j}([0,1])} = \emptyset,
	\end{align}
for all $i, j \in \{ 1, 2, \dots, n\}$ with $i \neq j$; in the case that $\beta \neq 2^{1/n}$ and $\alpha$ is the singleton in $I_{n,k}(\beta)$ the intersection in \eqref{eq:intersection_transitive_empty} is a singleton when $n \neq 2$ and a two point set when $n = 2$.  These observations in tandem with \Cref{SFTdense} and \Cref{cor:sofic} directly yield \Cref{thm:thmC}.  In order to prove \Cref{thm:thmD}, we require one final result. 

\begin{theorem}[\cite{Par:1979,MR0166332}]
Let $(\beta, \alpha) \in \Delta$  be fixed.  The absolutely continuous measure $\nu_{\beta, \alpha}$ with density
	\begin{align*}
	h_{\beta, \alpha} \coloneqq \sum_{n = 0}^{\infty} \beta^{-n} \left (\mathds{1}_{[0, (T_{\beta,\alpha}^{+})^{n}(1))} - \mathds{1}_{[0, (T_{\beta,\alpha}^{+})^{n}(0))} \right)
	\end{align*}
is invariant under $T^{\pm}$.  Moreover, the support of $\nu_{\beta, \alpha}$ equals $[0, 1]$ and only if $(\beta, \alpha) \not\in D_{n, k}$  or if $\beta = 2^{1/n}$ and $\alpha$ is the single point of $I_{n, k}(2^{1/n})$, for some $n, k  \in \mathbb{N}$ with $k < n$ and $\mathrm{gcd}(k, n) = 1$.  Further, in the case that $\beta \neq 2^{1/n}$ and $(\beta, \alpha) \in D_{n, k}$, the support of $\nu_{\beta, \alpha}$ is contained in the disjoint union of intervals,
	\begin{align*}
	\bigcup_{i = 1}^{n} \overline{(T^{\pm}_{\beta, \alpha})^{i}([0, 1])}.
	\end{align*}
\end{theorem}

\Cref{thm:thmD} follows from this result in tandem with the observations directly proceeding it together with \Cref{thm:thmB}.

\bibliographystyle{plain}
\bibliography{fiberdensityref}

\begin{thebibliography}{10}

\bibitem{AFPK}
S.~Akiyama, J.~Feng, D.\, T.~Kempton, and T.~Persson.
\newblock On the {H}ausdorff dimension of {B}ernoulli convolutions.
\newblock {\em International Math. Res. Notices}, 2018.

\bibitem{MR2299792}
J.-P. Allouche, C.~Frougny, and K.~G. Hare.
\newblock On univoque {P}isot numbers.
\newblock {\em Math. Comp.}, 76(259):1639--1660, 2007.

\bibitem{MR3683942}
S.~Baker, Z.~Mas\'{a}kov\'{a}, E.~Pelantov\'{a}, and T.~V\'{a}vra.
\newblock On periodic representations in non-{P}isot bases.
\newblock {\em Monatsh. Math.}, 184(1):1--19, 2017.

\bibitem{BHV}
M.~Barnsley, B.~Harding, and A.~Vince.
\newblock The entropy of a special overlapping dynamical system.
\newblock {\em Ergodic Theory Dynam.\ Systems}, 34(2):483--500, 2014.

\bibitem{MR0447134}
A.~Bertrand.
\newblock D\'eveloppements en base de {P}isot et r\'epartition modulo {$1$}.
\newblock {\em C. R. Acad. Sci. Paris S\'er. A-B}, 285(6):A419--A421, 1977.

\bibitem{MR1024551}
D.~W. Boyd.
\newblock Salem numbers of degree four have periodic expansions.
\newblock In {\em Th\'eorie des nombres ({Q}uebec, {PQ}, 1987)}, pages 57--64.
  de Gruyter, Berlin, 1989.

\bibitem{MR1333306}
D.~W. Boyd.
\newblock On the beta expansion for {S}alem numbers of degree {$6$}.
\newblock {\em Math. Comp.}, 65(214):861--875, $S$29--$S$31, 1996.

\bibitem{MR1483916}
D.~W. Boyd.
\newblock The beta expansion for {S}alem numbers.
\newblock In {\em Organic mathematics ({B}urnaby, {BC}, 1995)}, volume~20 of
  {\em CMS Conf. Proc.}, pages 117--131. Amer. Math. Soc., 1997.

\bibitem{KalleBruinCarminati}
H.~Bruin, C.~Carminati, and C.~Kalle.
\newblock Matching for generalised {$\beta$}-transformations.
\newblock {\em Indag.\ Math.}, 28(1):55--73, 2017.

\bibitem{0951-7715-32-1-172}
H.~Bruin, C.~Carminati, S.~Marmi, and A.~Profeti.
\newblock Matching in a family of piecewise affine maps.
\newblock {\em Nonlinearity}, 32(1), 2019.

\bibitem{COOPERBAND202096}
Z.~Cooperband, E.~P.~J. Pearse, B.~Quackenbush, J.~Rowley, T.~Samuel, and
  M.~West.
\newblock Continuity of entropy for lorenz maps.
\newblock {\em Indag.\ Math.\ (N.S.)}, 31(1):96--105, 2020.

\bibitem{MR3084706}
K.~Dajani and C.~Kalle.
\newblock Local dimensions for the random {$\beta$}-transformation.
\newblock {\em New York J. Math.}, 19:285--303, 2013.

\bibitem{DK:2002b}
K.~Dajani and C.~Kraaikamp.
\newblock From greedy to lazy expansions and their driving dynamics.
\newblock {\em Expo.\ Math.}, 20(4):315--327, 2002.

\bibitem{1011470}
I.~Daubechies, R.~DeVore, C.~S. Gunturk, and V.~A. Vaishampayan.
\newblock Beta expansions: a new approach to digitally corrected a/d
  conversion.
\newblock In {\em 2002 IEEE International Symposium on Circuits and Systems.
  Proceedings}, volume~2, pages II--784--II--787, 2002.

\bibitem{MR1078082}
P.~Erd\"os, I.~Jo\'o, and V.~Komornik.
\newblock Characterization of the unique expansions
  {$1=\sum^\infty_{i=1}q^{-n_i}$} and related problems.
\newblock {\em Bull. Soc. Math. France}, 118(3):377--390, 1990.

\bibitem{G:1990}
P.~Glendinning.
\newblock Topological conjugation of lorenz maps by $\beta$-transformations.
\newblock {\em Math.\,Proc.\,Camb.\,Phil.\,Soc.}, 107:401--413, 1990.

\bibitem{MR1399483}
P.~Glendinning and T.~Hall.
\newblock Zeros of the kneading invariant and topological entropy for {L}orenz
  maps.
\newblock {\em Nonlinearity}, 9(4):999--1014, 1996.

\bibitem{MR0386019}
S.~Halfin.
\newblock Explicit construction of invariant measures for a class of continuous
  state {M}arkov processes.
\newblock {\em Ann. Probability}, 3(5):859--864, 1975.

\bibitem{H}
F.~Hofbauer.
\newblock Maximal measures for piecewise monotonically increasing
  transformations on {$[0,1]$}.
\newblock In {\em Ergodic theory ({P}roc. {C}onf., {M}ath. {F}orschungsinst.,
  {O}berwolfach, 1978)}, volume 729, pages 66--77. Lecture Notes in Math.,
  1979.

\bibitem{MR570882}
F.~Hofbauer.
\newblock On intrinsic ergodicity of piecewise monotonic transformations with
  positive entropy.
\newblock {\em Israel J. Math.}, 34(3):213--237 (1980), 1979.

\bibitem{MR599481}
F.~Hofbauer.
\newblock On intrinsic ergodicity of piecewise monotonic transformations with
  positive entropy. {II}.
\newblock {\em Israel J. Math.}, 38(1-2):107--115, 1981.

\bibitem{HS}
J.~H. Hubbard and C.~T. Sparrow.
\newblock The classification of topologically expansive {L}orenz maps.
\newblock {\em Comm.\ Pure Appl.\ Math.}, 43(4):431--443, 1990.

\bibitem{MR0346134}
S.~Ito and Y.~Takahashi.
\newblock Markov subshifts and realization of {$\beta $}-expansions.
\newblock {\em J.\ Math.\ Soc.\ Japan}, 26:33--55, 1974.

\bibitem{KalleSteiner}
C.~Kalle and W.~Steiner.
\newblock Beta-expansions, natural extensions and multiple tilings associated
  with {P}isot units.
\newblock {\em Trans.\ Amer.\ Math.\ Soc.}, 364(5):2281--2318, 2012.

\bibitem{Komornik:2011}
V.~Komornik.
\newblock Expansions in noninteger bases.
\newblock {\em Integers}, 11B:Paper No. A9, 30, 2011.

\bibitem{Komornik:1998}
V.~Komornik and P.~Loreti.
\newblock Unique developments in non-integer bases.
\newblock {\em American Math. Monthly}, 105:636--639, 1998.

\bibitem{LSS}
B.~Li, T.~Sahlsten, and T.~Samuel.
\newblock Intermediate {$\beta$}-shifts of finite type.
\newblock {\em Discrete\,Contin.\,Dyn.\,Syst.}, 36(1):323--344, 2016.

\bibitem{LSSW}
B.~Li, T.~Sahlsten, T.~Samuel, and W.~Steiner.
\newblock Denseness of intermediate $\beta$-shifts of finite type.
\newblock {\em Proc.\ Amer.\ Math.\ Soc.}, 147(5):2045--2055, 2019.

\bibitem{Lind:84}
D.~A. Lind.
\newblock The entropies of topological {M}arkov shifts and a related class of
  algebraic integers.
\newblock {\em Ergodic Theory Dynam. Systems}, 4(2):283--300, 1984.

\bibitem{Maia_18}
B.~M. Maia.
\newblock The beta-transformation's companion map for pisot or salem numbers
  and their periodic orbits.
\newblock {\em Dynamical Systems}, 33(1):1--9, 2018.

\bibitem{ArneThesis}
A.~Mosbach.
\newblock Finite and infinite rotation sequences and beyond, 2019.
\newblock PhD Thesis Universit\"at Bremen.

\bibitem{Par:1979}
R.~Palmer.
\newblock On the classification of measure preserving transformations of
  lebesgue spaces, 1979.
\newblock PhD Thesis University of Warwick.

\bibitem{MR0142719}
W.~Parry.
\newblock On the {$\beta $}-expansions of real numbers.
\newblock {\em Acta Math.\ Acad.\ Sci.\ Hungar.}, 11:401--416, 1960.

\bibitem{MR0166332}
W.~Parry.
\newblock Representations for real numbers.
\newblock {\em Acta Math.\ Acad.\ Sci.\ Hungar.}, 15:95--105, 1964.

\bibitem{Parry:1979}
W.~Parry.
\newblock The {L}orenz attractor and a related population model.
\newblock In {\em Ergodic theory ({P}roc. {C}onf., {M}ath. {F}orschungsinst.,
  {O}berwolfach, 1978)}, volume 729 of {\em Lecture Notes in Math.}, pages
  169--187. Springer, Berlin, 1979.

\bibitem{R:1957}
A.~R\'enyi.
\newblock Representations for real numbers and their ergodic properties.
\newblock {\em Acta\,Math.\,Acad.\,Sci.\,Hungar}, 8:477--493, 1957.

\bibitem{KS}
K.~Schmidt.
\newblock
  On\,periodic\,expansions\,of\,{P}isot\,numbers\,and\,{S}alem\,numbers.
\newblock {\em Bull.\,London\,Math.\,Soc.}, 12(4):269--278, 1980.

\bibitem{Sidorov:2003}
N.~Sidorov.
\newblock Almost every number has a continuum of $\beta$-expansions.
\newblock {\em American Math. Monthly}, 110:838--842, 2003.

\bibitem{MR2052279}
N.~Sidorov.
\newblock Arithmetic dynamics.
\newblock In {\em Topics in dynamics and ergodic theory}, volume 310 of {\em
  London Math. Soc. Lecture Note Ser.}, pages 145--189. Cambridge Univ. Press,
  Cambridge, 2003.

\bibitem{MR681294}
C.~Sparrow.
\newblock {\em The {L}orenz equations: bifurcations, chaos, and strange
  attractors}, volume~41 of {\em Applied Mathematical Sciences}.
\newblock Springer-Verlag, New York-Berlin, 1982.

\end{thebibliography}

\end{document}